\documentclass[12pt]{amsart}
\usepackage{bbm}
\usepackage{mathrsfs}
\usepackage[english]{babel}
\usepackage{wrapfig}

\usepackage[top=3cm,bottom=2cm,left=3cm,right=3cm,marginparwidth=1.75cm]{geometry}

\usepackage{amsmath,amssymb}
\usepackage{graphicx}
\usepackage{cite}
\usepackage{enumitem} 
\usepackage[all,cmtip]{xy}
\usepackage{tikz-cd}
\usetikzlibrary{intersections}
\usepackage{relsize}
\usepackage{faktor}
\usepackage{dsfont}
\usepackage{xcolor, import}
\usepackage{bm}
\usepackage{hyperref}
\hypersetup{
 citebordercolor={green!40!black},
 linkbordercolor={red!50!black},
 urlbordercolor={blue!70!black}
}

\makeatletter
\renewcommand{\section}{\@startsection%
{section}
{1}
{0mm}
{1.5\bigskipamount}
{0.5\bigskipamount}
{\centering\normalsize\sc}}

\renewcommand{\subsection}{\@startsection%
{subsection}
{2}
{0mm}
{0.5\bigskipamount}
{0.5mm}
{\normalsize\sc}}

\renewcommand{\paragraph}{\@startsection%
{paragraph}
{4}
{0mm}
{\bigskipamount}
{0pt}
{\normalsize\bf}}

\expandafter\let\expandafter\oldproof\csname\string\proof\endcsname
\let\oldendproof\endproof
\renewenvironment{proof}[1][\proofname]{%
  \oldproof[\slshape #1]%
}{\oldendproof}
\def\provedboxcontents#1{$\square$}

\newtheoremstyle{thm}{6pt plus 1pt minus 1pt}{6pt plus 1pt minus 1pt}{\slshape}{}{\scshape}{.}{5pt plus 1pt minus 1pt}{}
\newtheoremstyle{def}{6pt plus 1pt minus 1pt}{6pt plus 1pt minus 1pt}{}{}{\scshape}{.}{5pt plus 1pt minus 1pt}{}
\newtheoremstyle{rmk}{6pt plus 1pt minus 1pt}{6pt plus 1pt minus 1pt}{}{}{\scshape}{.}{5pt plus 1pt minus 1pt}{}
\newtheoremstyle{claim}{6pt plus 1pt minus 1pt}{6pt plus 1pt minus 1pt}{}{}{\slshape}{.}{5pt plus 1pt minus 1pt}{}

\theoremstyle{thm}
\newtheorem{newstatement}{newstatement}
\newtheorem{lemma}[newstatement]{Lemma}

\newtheorem*{theorem*}{Theorem 2}

\newtheorem*{stiefel-thm}{Stiefel's Parallelizability Theorem}

\theoremstyle{def}

\theoremstyle{rmk}
\newtheorem{remark}[newstatement]{Remark}

\newtheorem*{example*}{Example}

\theoremstyle{claim}

\DeclareMathOperator{\SO}{SO}

\renewcommand{\epsilon}{\varepsilon}
\renewcommand{\phi}{\varphi}

\renewcommand{\:}{\,{:}\;}

\newcommand{\R}{\mathbb{R}}

\newcommand{\Z}{\mathbb{Z}}

\newcommand{\Bd}{\partial}

\DeclareMathOperator{\id}{id}
\DeclareMathOperator{\cs}{\#}

\let\emph\textsl
\title{On Stiefel's parallelizability of 3-manifolds}

\author{Valentina Bais}
\address{Dipartimento di Matematica e Geoscienze, Università di Trieste, Via Valerio 12/1, 34127 Trieste, Italy.}
\email{valentina.bais@studenti.units.it}

\author{Daniele Zuddas}
\address{Dipartimento di Matematica e Geoscienze, Università di Trieste, Via Valerio 12/1, 34127 Trieste, Italy.}
\email{dzuddas@units.it}

\date{}

\begin{document}


\begin{abstract}
We give a new elementary proof of the parallelizability of closed orientable 3-manifolds. We use as the main tool the fact that any such manifold admits a Heegaard splitting.
\end{abstract}


\maketitle

The aim of this note is to present a new, elementary proof of the following classical theorem \cite{St1935}.
\begin{stiefel-thm}
Every smooth closed orientable 3-manifold is parallelizable.
\end{stiefel-thm}

We recall that a smooth $m$-manifold $M$ is said to be \textsl{parallelizable} if its tangent bundle $TM$ is trivial or, equivalently, if there are $m$ vector fields on $M$ which are everywhere linearly independent. Such an $m$-tuple of vector fields is said to be a \textsl{parallelization} or \textsl{frame field} on $M$. Notice that if $M$ is parallelizable then its Euler characteristic vanishes, so the only closed connected parallelizable 2-manifold is the torus. 

In literature there are other several elementary (as well as less elementary) proofs of Stiefel's Parallelizability Theorem, see for example Benedetti and Lisca \cite{BL2018}, Durst, Geiges, Gonzalo and Kegel \cite{DGGK2020}, Gonzalo \cite{Go1987}, Kirby \cite[Chapter VII]{K1989}, Geiges \cite[Section 4.2]{Ge2008}, Fomenko and Matveev \cite[Section 9.4]{FM1997} and Whitehead \cite{Wh1961}. However, as far as we know, there is no trace in literature of the proof given in the present paper. We believe that our proof uses very minimal background, such as some basic facts about Morse theory, vector bundles, homology and linear algebra, and should be at the level of university master students.

Hereafter manifolds, submanifolds and maps between them will be smooth, if not differently stated.

\begin{example*}
The linear vector fields
\begin{align*}
&u_1 = -x_2\frac{\partial}{\partial x_1} + x_1\frac{\partial}{\partial x_2} - x_4\frac{\partial}{\partial x_3} + x_3\frac{\partial}{\partial x_4},\\
&u_2 = -x_3\frac{\partial}{\partial x_1} + x_4\frac{\partial}{\partial x_2} + x_1\frac{\partial}{\partial x_3} - x_2\frac{\partial}{\partial x_4},\\
&u_3 = -x_4\frac{\partial}{\partial x_1} - x_3\frac{\partial}{\partial x_2} + x_2\frac{\partial}{\partial x_3} + x_1\frac{\partial}{\partial x_4}
\end{align*}
define an orthonormal parallelization of the unit sphere $S^3$. They are obtained by quaternion multiplication on the left by $i$, $j$ and $k$ respectively, with $x_1 + x_2 i +x_3 j + x_4 k \in S^3 \subset H \cong R^4$. Since $u_1$, $u_2$ and $u_3$ are invariant with respect to the antipodal map $-\id_{S^3}$, they pass to the quotient yielding a parallelization of $RP^3 = S^3/\{\pm \id_{S^3}\}$.
\end{example*}

\begin{proof}[Proof of the parallelizability theorem]
Let $M$ be a closed connected oriented 3-man\-i\-fold. We want to prove that there are three vector fields $(w_1,w_2,w_3)$ on $M$ which are linearly independent at every point.

The manifold $M$ admits a \textsl{Heegaard splitting}, namely a splitting of the form \[M = M'\cup M'',\] where $M'$ and $M''$ are 3-dimensional handlebodies contained in $M$ of the same genus $g\geq 0$, and \[F = \Bd M' = \Bd M'' = M'\cap M''\] is a closed connected orientable surface of genus $g$.

Clearly, the handlebodies $M'$ and $M''$ can be embedded in $R^3$ (in a standard way) and so each one of them is parallelizable, by restricting the canonical frame field of $R^3$. In this way, we get two parallelizations $V' = (v_1', v_2', v_3')$ and $V'' = (v_1'', v_2'', v_3'')$ on $M'$ and $M''$ respectively. If $V'$ and $V''$ agree along the Heegaard surface $F$, then they can be glued together to give a parallelization of $M$. However, in general this will not be the case. 

Let us fix a Riemannian metric on $M$. By Gram-Schmidt orthogonalization we can assume that $V'$ and $V''$ are orthonormal frame fields that agree with the given orientation of $M$.

Let $A\: F \to \SO(3)$ be the change of basis matrix map from the basis $V'_{|F}$ to $V''_{|F}$. If $A$ were null-homotopic, then $V'$ could be continuously changed into $V''$ in a tubular neighborhood $U \cong F \times [0,1]$ of $F$, thus yielding a parallelization of $M$. Observe that a map $A\: F \to \SO(3)$ defined on a surface $F$ is null-homotopic if and only if the induced homomorphism between fundamental groups is trivial, since this is exactly the condition for finding a lift to the universal covering $S^3 \to \SO(3)$, and every map $F \to S^3$ is clearly null-homotopic. So, we are left to show that the parallelizations $V'$ and $V''$ on $M'$ and $M''$, respectively, can be suitably chosen so that the corresponding change of basis matrix map $A$ is null-homotopic.

Notice that for every path-connected space $X$ and for every continuous base point-preserving map $f\: X \to \SO(3)$, the induced homomorphism $f_* \: \pi_1(X) \to \pi_1(\SO(3)) \cong \Z_2$ is trivial if and only if so is the induced linear map of $\Z_2$-vector spaces\break $f_* \: H_1(X;\Z_2) \to H_1(\SO(3); \Z_2) \cong \Z_2$, as it can be immediately realized by considering the Hurewicz homomorphism $\pi_1(X) \to H_1(X; \Z)$ and the coefficient homomorphism $H_1(X; \Z) \to H_1(X;\Z_2)$. Hereafter, by $f_*$ we mean the map induced in first homology with $\Z_2$ coefficients.

Next consider any two arbitrary orthonormal positive frame fields $W' = (w_1',w_2',w_3')$ on $M'$ and $W'' = (w_1'',w_2'',w_3'')$ on $M''$. Then we have change of basis matrix maps $C' \: M' \to \SO(3)$ from the basis $W'$ to $V'$ and $C'' \: M'' \to \SO(3)$ from $W''$ to $V''$, as well as the change of basis matrix map $B \: F \to \SO(3)$ from the basis $W'_{|F}$ to $W''_{|F}$. Therefore, we obtain \[B = (C''_{|F})^{-1}\cdot A \cdot C'_{|F}.\] An easy computation yields
\begin{equation}\label{B*/eqn}
B_* = (C''_{|F})_* + A_* + (C'_{|F})_* \: H_1(F;\Z_2) \to H_1(\SO(3); \Z_2).
\end{equation}
Here we are using the elementary fact that for every path-connected Lie group $(G, \cdot)$ ($\SO(3)$ in our case), for every path-connected topological space $X$ and for every continuous maps $f, g \: X \to G$ we have \[(f\cdot g)_* = f_* + g_* \: H_1(X; \Z_2) \to H_1(G; \Z_2)\]
where by $f \cdot g: X \to G$ we denote the map defined by $(f \cdot g )(x):=f(x) \cdot g(x)$ for every $x \in X$ (this can be easily shown by first proving the above equality at the fundamental group level, and then using the Hurewicz homomorphism and coefficient reduction modulo two).
So, we are left to show that there exist maps $C'$ and $C''$ as above such that $B_* = 0$. It is actually enough to construct the linear maps $C'_*$ and $C''_*$ in homology. Indeed, being $M'$ and $M''$ handlebodies, they deformation retract to a bouquet of $g$ circles $M'\cong M'' \simeq \vee_{\!g} S^1$, and so every linear function $\phi \: H_1(M';\Z_2) \to \Z_2$ is induced by a map $M'\to \SO(3)$ obtained by composing a homotopy equivalence $r\: M' \to \vee_{\!g} S^1$ with a map $\vee_{\!g} S^1 \to \SO(3)$ that suitably sends the $i$-th circle $S^1_i \subset \vee_{\!g} S^1$ to the identity matrix if $\phi(r_*^{-1}([S^1_i])) = 0$ or to $\SO(2)\subset \SO(3)$ homeomorphically if $\phi(r_*^{-1}([S^1_i])) = 1$, for every $i = 1,\dots, g$ (and similarly for $M''$). 

Next we observe that every $\alpha \in H_1(F;\Z_2)$ can be represented by a closed connected simple curve $a \subset F$. This is obvious if $\alpha = 0$ as it is the homology class of a circle in $F$ which is the boundary of a disk. On the other hand, if $\alpha\neq 0$ we can consider a canonical symplectic basis $\mu_1, \lambda_1, \dots, \mu_g, \lambda_g$ of $H_1(F;\Z_2)$, each element of which is so represented, and write $\alpha = \mu_{i_1} + \cdots +\mu_{i_r} + \lambda_{j_1} + \cdots +\lambda_{j_s}$ for some $i_1< \cdots<i_r$ and $j_1<\cdots < j_s$. Whenever $i_p = j_q$, the term $\mu_{i_p} + \lambda_{j_q}$ can be represented by a closed connected simple curve obtained by the usual crossing desingularization at a single transversal intersection point of the two simple curves representing $\mu_{i_p}$ and $\lambda_{j_q}$, and so $\alpha$ turns out to be the homology class of the disjoint union of finitely many closed simple curves. It is now enough to join them by suitable pairwise disjoint embedded bands.

Let $i'\: F \to M'$ and $i'': F \to M''$ be the inclusion maps. For every $\alpha \in H_1(F;\Z_2) \cong \Z_2^{2g}$, we set $\alpha' := i'_*(\alpha) \in H_1(M';\Z_2) \cong \Z_2^{g}$ and $\alpha'' := i''_*(\alpha) \in H_1(M'';\Z_2) \cong \Z_2^{g}$. Notice that $i'_*$ and $i''_*$ are surjective and so $\ker(i'_*)$ and $\ker(i''_*)$ are linear subspaces of $H_1(F;\Z_2)$ of dimension $g$.

We now prove that $A_*(\alpha) = 0$ for every $\alpha\in \ker(i'_*) \cap \ker(i''_*)$. Let $a\subset F$ be a closed connected simple curve representing $\alpha$. Then, there exist compact connected properly embedded surfaces $S'\subset M'$ and $S''\subset M''$ with $\Bd S'=\Bd S'' = a$, such that $S = S'\cup S''$ is a closed smooth surface in $M$.

In the following lemma we show that $S$ admits a parallelizable tubular neighborhood $U$ in $M$.
We then have $A_{|F\cap U}=(D''_{|F \cap U})^{-1} \cdot D'_{|F\cap U}$,  where $D'$ and $D''$ are the change of basis matrix maps from $V'_{|M'\cap U}$ and $V''_{|M''\cap U}$, respectively, to a fixed positive orthonormal parallelization of $U$. This will allow us to conclude, since $$A_*(\alpha)=(D''_{|F \cap U})_*(\alpha) + (D'_{|F \cap U})_*(\alpha)=D''_*(\alpha'') + D'_*(\alpha')=0,$$ since $\alpha' = 0$ in $H_1(M'\cap U;\Z_2)$ and $\alpha'' = 0$ in $H_1(M''\cap U;\Z_2)$.

Consider a basis $\alpha_1,\dots, \alpha_k$ for $\ker(i'_*) \cap \ker(i''_*)$, for some $k\geq 0$, and extend it to a basis of $\ker(i'_*)$ by the classes $\beta_1, \dots \beta_h$ and to a basis of $\ker(i''_*)$ by the classes $\gamma_1, \dots \gamma_h$, with $h = g-k$. Then, $\alpha_1,\dots, \alpha_k, \beta_1, \dots \beta_h, \gamma_1, \dots, \gamma_h$ are linearly independent as they form a basis of $\ker(i'_*) + \ker(i''_*)$, and so they can be extended to a basis of $H_1(F;\Z_2)$ by the classes $\delta_1,\dots, \delta_k$. It follows that $\gamma_1',\dots, \gamma_h', \delta_1',\dots, \delta_k'$ form a basis of $H_1(M';\Z_2)$, and $\beta_1'',\dots, \beta_h'', \delta_1'',\dots, \delta_k''$ form a basis of $H_1(M'';\Z_2)$.

We have already proved that $B_*(\alpha_j)=0$ for every $j=1,\dots,k$, and for every choice of the change of basis matrices $C'$ and $C''$. 

For every choice of the indices we set
\begin{align}
\begin{split}\label{C*/eqn}
& C'_*(\gamma'_j) := A_*(\gamma_j),\qquad C'_*(\delta'_j) := A_*(\delta_j),\\
& C''_*(\beta''_j) := A_*(\beta_j),\qquad C''_*(\delta''_j) := 0.
\end{split}
\end{align}
Then we get linear functions $C'_* \: H_1(M';\Z_2) \to \Z_2$ and $C''_*\: H_1(M'';\Z_2) \to \Z_2$, where the identification $H_1(\SO(3);\Z_2) \cong \Z_2$ is understood. Such linear functions are then induced by smooth maps $C'\: M' \to \SO(3)$ and $C''\: M'' \to \SO(3)$, respectively.

By equations \ref{B*/eqn} and \ref{C*/eqn} above we then obtain $B_*=0$. This, together with the following lemma, is enough to conclude the proof.
\end{proof}

\begin{lemma}
Let $M$ be an orientable 3-manifold and let $S \subset M$ be a smooth closed connected surface. Then $S$ has a parallelizable tubular neighborhood.
\end{lemma}

\begin{proof}
Since every open tubular neighborhood of $S$ in $M$ is diffeomorphic to the total space of the normal bundle $\nu_S$ of $S$ in $M$, it will be enough to prove that $\nu_S$ is determined, up to bundle isomorphisms, only by $S$, namely it is independent of the embedding of $S$ in $M$, and that every closed surface can be embedded in a certain parallelizable 3-manifold $N$. This implies that a tubular neighborhood of $S$ in $M$ is diffeomorphic to a tubular neighborhood of $S$ in $N$, which is parallelizable.

If $S$ is orientable, then  $\nu_S$ is trivial (hence it is independent of the embedding) because, by means of an orientation of $S$ and of $M$, a unit normal vector field can be constructed along $S$. Moreover, $S$ embeds in $R^3$, which is parallelizable.

If $S$ is non-orientable, then it is a connected sum of $n$ copies of $RP^2$, that is\break $S \cong \cs_n RP^2$ for some $n \geq 1$, and so it can be embedded in $N = \cs_n RP^3$, which is a parallelizable $3$-manifold. Indeed, $RP^3$ is parallelizable as it is shown in the example above. Moreover, if $M_1$ and $M_2$ are oriented parallelizable 3-manifolds, then also their connected sum $M_1 \cs M_2$ is parallelizable, because any two positive parallelizations of $M_1$ and $M_2$ can be homotoped to coincide in a tubular neighborhood of the 2-sphere along which the connected sum is made.

We are left to prove that $\nu_S$ depends only on $S$ (namely, on the number $n$ of the $RP^2$ connected summands in our situation). The total space of $\nu_S$, being diffeomorphic to a tubular neighborhood of $S$ in $M$, is orientable. Hence, in order to conclude, it is enough to show that there exists a unique real line bundle $\xi\: E \to S$  (up to bundle isomorphisms) whose total space $E$ is orientable.

Fix a Riemannian metric on the bundle $\xi$ and consider the following commutative diagram:
\[\begin{tikzcd}
	\widehat E & E \\
	{\widehat S} & S
	\arrow["2\,\mathbin{:}\,1", "p"', from=2-1, to=2-2]
	\arrow["\xi", from=1-2, to=2-2]
	\arrow["{\widehat\xi}"', from=1-1, to=2-1]
	\arrow["2\,\mathbin{:}\,1"', "{\widehat p}", from=1-1, to=1-2]
\end{tikzcd}\]
where $\widehat S$ is the closed connected orientable surface of genus $n-1$, $p\: \widehat S \to S$ is the orientable double covering of $S$, $\widehat\xi = p^*(\xi) \: \widehat E \to \widehat S$ is the pullback bundle of $\xi$ via $p$ and $\widehat p\: \widehat E \to E$ is the induced fiber-preserving double covering. Since $E$ is orientable, so is $\widehat E$ and by the orientable case above we can identify $\widehat E$ with $\widehat S \times \R$ by means of a suitable bundle isomorphism, which can be assumed to be unitary with respect to the Riemannian metric induced by $\widehat p$ on $\widehat\xi$.

The non-trivial covering transformation of $p$ is an orientation-reversing free involution $\iota: \widehat S \to \widehat S$. On the other hand, the free involution $\hat \iota$ of $\widehat E \cong \widehat S\times \R$ upstairs, determined by $\widehat p$, is orientation-preserving, and moreover $\widehat\xi\circ \hat\iota = \iota\circ \widehat\xi$. Then $\hat\iota\: \widehat S \times \R \to \widehat S \times \R$ is forced to be the map 
\[\hat\iota(x,t) = (\iota(x),-t)\]
and hence $E \cong \widehat E/\langle\hat\iota\rangle$ is uniquely determined up to bundle isomorphisms.
\end{proof}

\begin {remark}
Over a circle there are only two isomorphism classes of line bundles, namely the trivial one and the non-orientable one. So, for the non-orientable case in the previous lemma, one could also argue as follows. A line bundle $\xi$ over a surface $S$ depends only on its restriction over the 1-skeleton (a bouquet of circles), up to bundle isomorphisms. Assume that the total space is orientable. Therefore, over every circle in the 1-skeleton, the restriction of $\xi$ must be non-trivial (hence non-orientable) exactly when that circle has non-trivial normal bundle in $S$. This determines uniquely the bundle.
\end{remark}

\begin{remark}
More generally, one can consider a smooth oriented rank-3 real vector bundle $\xi\: E \to M$ (endowed with a Riemannian metric) and, by taking orthonormal positive trivializations over $M'$ and $M''$, there is a change of basis matrix map\break $A \: F \to \SO(3)$.
The Mayer-Vietoris homology exact sequence applied to the Heegaard splitting $M = M' \cup M''$ tells us that the boundary map is an isomorphism between $H_2(M;\Z_2)$ and $\ker(i'_*) \cap \ker(i''_*)$ and, moreover, the restriction $A_{*|} \: \ker(i'_*) \cap \ker(i''_*) \to \Z_2$ can be identified with the second Stiefel-Whitney class $w_2(\xi) \in H^2(M;\Z_2) = H_2(M;\Z_2)^*$. Then, the argument used in the proof above yields a Heegaard splitting based interpretation of the well-known obstruction-theoretic fact that an orientable rank-3 vector bundle $\xi$ over a closed orientable 3-manifold $M$ is trivial if and only if $w_2(\xi)=0$.
\end{remark}

\section*{Acknowledgements}
We would like to thank the anonymous referee for his or her suggestions that have been useful to improve the manuscript.

The second author is member of GNSAGA -- Istituto Nazionale di Alta Matematica `Francesco Severi', Italy.


\end{document}